\documentclass[12pt]{article}
\title{On the least square-free primitive root modulo $p$}
\author{Stephen D. Cohen \\
  School of Mathematics and Statistics, \\
  University of Glasgow, Scotland \\
  Stephen.Cohen@glasgow.ac.uk
and\\ \\
Tim Trudgian\footnote{Supported by Australian Research Council DECRA Grant DE120100173.}\\
Mathematical Sciences Institute\\ The Australian National University,
 ACT 0200, Australia\\ timothy.trudgian@anu.edu.au
}
\usepackage{url}
\usepackage{amsthm}
\usepackage{amsmath}
\usepackage{fullpage}
\usepackage{enumerate}
\usepackage{amssymb}

\usepackage{booktabs}
\usepackage[boxed]{algorithm2e}
\usepackage{float}

\newtheorem{thm}{Theorem}
\newtheorem{Lem}{Lemma}

\newtheorem{conj}{Conjecture}

\begin{document}
\maketitle
\begin{abstract}
\noindent
Let $g^{\square}(p)$ denote the least square-free  primitive root modulo $p$. We show that $g^{\square}(p)< p^{0.96}$ for all $p$.
\end{abstract}
\section{Introduction}
Let $\hat{g}(p)$ denote the least  prime primitive root modulo $p$. It is not known whether $\hat{g}(p)<p$ for all $p$, or even for all sufficiently large $p$. The best unconditional result is due to Ha \cite{Ha}, namely, that $\hat{g}(p) \ll p^{3.1}$. On the Generalised Riemann Hypothesis it is known \cite{Shoup} that $\hat{g}(p) \ll (\log p)^{6+ \epsilon}$, and, recently, it was shown in \cite{MTT} that $\hat{g}(p) < \sqrt{p} - 2$ for all $p> 2791$.

Rather than consider prime primitive roots, we consider the broader (and easier) case of square-free primitive roots. An integer $n$ is said to be \emph{square-free}  if for all primes $l|n$ we have $l^{2}\nmid n$. Let $g^{\square}(p)$ denote the least square-free primitive root modulo $p$,
and let $N^{\square}(p,x)$ denote the number of square-free primitive roots modulo $p$ that do not exceed $x$.
Shapiro \cite[p.\ 355]{Shapiro} showed that
\begin{equation}\label{one}
N^{\square}(p,x) = \frac{\phi(p-1)}{p-1} \left\{ \frac{6}{\pi^{2}}x + O( 2^{\omega(p-1)} p^{1/4} (\log p)^{1/2} x^{1/2}\right\}.
\end{equation}
This shows that $N^{\square}(p, p^{1/2 + \epsilon})>0$ for any positive $\epsilon$ and for all sufficiently large $p$. Equivalently, this means that $g(p) \ll p^{1/2 + \epsilon}$.

The error term in (\ref{one}) has been  improved by Liu and Zhang \cite[Thm 1.1]{LiuZhang05}, who showed
\begin{equation}\label{two}
N^{\square}(p,x) = \frac{\phi(p-1)}{p-1} \left\{ \frac{6}{\pi^{2}}x + O\left( p^{9/44 + \epsilon} x^{1/2 + \epsilon}\right) \right\},
\end{equation}
whence one has that $g^{\square}(p)\ll p^{9/22 + \epsilon}$. 
%indeed, assuming the generalised Lindel\"{o}f hypothesis for Dirichlet $L$-functions, the proof in \cite{LiuZhang05} gives $g^{\square}(p)\ll p^{\epsilon}$.
Rather than focus on (\ref{two}) we seek a version of (\ref{one}) in order to bound $g^{\square}(p)$ explicitly. We do this in the following theorem.
\begin{thm}\label{Main}
$g^{\square}(p) < p^{0.96}$ for all $p$. In particular all primes $p$  possess a square-free primitive root less than $p$.
\end{thm}
We note that using (\ref{one}) does not allow one to show that $g^{\square}(p) \ll p^{1/2}$. However, based on computational evidence, the bound in (\ref{two}) and recent work in \cite{COT, MTT} it seems reasonable to extrapolate, as below.
\begin{conj}
For all $p> 409$ we have $g^{\square}(p) < \sqrt{p} - 2$.
\end{conj}
The outline of this paper is as follows. In \S \ref{set} we collect the necessary results to make (\ref{one}) explicit. In \S \ref{sieve} we introduce a sieving inequality. We also carry out some rudimentary computations, which  prove Theorem \ref{Main}. Finally, in \S \ref{conch} we discuss a related problem on square-full primitive roots. Throughout this article we write $n= \square-\textrm{free}$ to indicate that $n$ is a square-free integer.

\section{Preliminary Results}\label{set}
The following establishes an indicator function on primitive roots.
\begin{equation}\label{confidence}
f(x):= \frac{\phi(p-1)}{p-1} \sum_{d|p-1} \frac{\mu(d)}{\phi(d)} \sum_{x\in \Gamma_{d}} \chi(n) = \begin{cases} 1 &\textrm{if $n$ is a primitive root mod $p$}\\  0 & \textrm{otherwise}.
\end{cases}
\end{equation}
We therefore have
\begin{equation}\label{ramsey}
\begin{split}
N^{\square}(p,x) &= \sum_{\substack{n\leq x\\ n= \square-\textrm{free}}} f(x)\\
&= \frac{\phi(p-1)}{p-1} \left\{ \sum_{\substack{n\leq x \\ n= \square-\textrm{free}}}1 + \sum_{\substack{d|p-1\\ d>1}} \frac{\mu(d)}{\phi(d)} \sum_{x\in\Gamma_{d}} \sum_{\substack{n\leq x\\ n= \square-\textrm{free}}} \chi(n)\right\}.
\end{split}
\end{equation}
Now, to estimate the inner-most sum in (\ref{ramsey}) we write
\begin{equation*}
\sum_{\substack{n\leq x \\n= \square-\textrm{free}}} \chi(n) = \sum_{n\leq x} \chi(n) \sum_{d^{2}|n} \mu(d)
= \sum_{d\leq \sqrt{x}} \mu(d) \sum_{\substack{n\leq x \\ n\equiv 0\, (\mathrm{mod}\ d^{2})}} \chi(n).
\end{equation*}
Upon taking absolute values and using the bound
\begin{equation}\label{clement} 
\bigg|\sum_{\substack{n\leq x \\ n\equiv 0\, (\mathrm{mod}\ d^{2})}} \chi(n)\bigg| \leq \min\left\{ \frac{x}{d^{2}}, \sqrt{p} \log p\right\},
\end{equation}
we find that
\begin{equation}\label{attlee}
\bigg|\sum_{\substack{n\leq x\\ n= \square-\textrm{free}}} \chi(n)\bigg| \leq 2 x^{1/2} p^{1/4} (\log p)^{1/2}.
\end{equation}
Finally, we need an estimate on the number of square-free numbers not exceeding $x$, which we borrow from \cite[Lemma 4.2]{Cipu}.
\begin{Lem}[Cipu]\label{james}
For all $x\geq 1$ we have
\begin{equation*}
\sum_{\substack{n\leq x\\ n= \square-\textrm{free}}}1 \geq \frac{6}{\pi^{2}} x - 0.104 \sqrt{x}.
\end{equation*}
\end{Lem}
While sharper estimates are known for $x\geq x_{0}>1$, the bound in Lemma \ref{james} is sufficient for our purposes.
Therefore, by (\ref{ramsey}) we have that $N^{\square}(p,x)>0$ if
\begin{equation}\label{callaghan}
G(x):= x^{1/2} p^{-1/4}-\frac{\pi^{2}}{6} \left(\frac{0.104}{p^{1/4}} + 2^{\omega(p-1) +1}(\log p)^{1/2}\right)>0.
\end{equation}
Setting $x= p^{0.96}$ in (\ref{callaghan}) shows that $G(p^{0.96})>0$ for all $p$ with $\omega(p-1) \geq 30$.  

Rather than consider all remaining cases of $\omega(p-1)$, we make use of the calculations in \cite[\S 4]{COT}, namely that
\begin{equation}\label{lloyd}
\hat{g}(p) < \sqrt{p} -2, \quad (2791<p < 2.5\cdot 10^{15}).
\end{equation} 
We use this to take care of the cases $1\leq \omega(p-1)\leq 7$. For example, when $\omega(p-1) = 7$ we find that $G(p^{0.96})>0$ for all $p> 5.5\cdot 10^{14}$. Hence we need only consider those $p\leq 5.5\cdot 10^{14}$, which are covered by (\ref{lloyd}). We continue in this way and dispatch the cases $1\leq \omega(p-1) \leq 7$.

\section{A sieving inequality and is consequences}\label{sieve}
\subsection{Sieving}
Given the prime $p$, let $e$ be any divisor of $p-1$.  Call an integer $n$ indivisible by $p$  \emph{$e$-free} if $n\equiv m^d (\mathrm{mod}\ p), d|e,$ for some integer $m$, implies $d=1$.
 (In particular, if $e=2$, be aware that, in this sense, being \emph{$2$-free} has a different meaning from being \emph{square-free}.) Observe that the definition $e$-free depends only on the distinct primes dividing $e$ and that, in particular, $n$ is a primitive root  if and only if $n$ is $(p-1)$-free.
  Extend the definition of the characteristic function (\ref{confidence}) as follows.
\begin{equation}\label{sconfidence}
f_e(x):= \frac{\phi(e)}{e} \sum_{d|e} \frac{\mu(d)}{\phi(d)} \sum_{x\in \Gamma_{d}} \chi(n) = \begin{cases} 1 &\textrm{if $n$ is $e$-free}\\  0 & \textrm{otherwise}.
\end{cases}
\end{equation}

Now, given $p$ and $x$ with $x<p$, let
 $N_e^{\square}(p,x)$ denote the number of square-free and $e$-free positive integers $n$ that do not exceed $x$.  Thus, $N^{\square}(p,x) =N_{p-1}^{\square}(x)$.
 From (\ref{sconfidence}) we have
 \begin{equation}\label{ramsay}
\begin{split}
N_e^{\square}(p,x) &= \sum_{\substack{n\leq x\\ n= \square-\textrm{free}}} f_e(x)\\
&= \frac{\phi(e)}{e} \left\{ \sum_{\substack{n\leq x \\n= \square-\textrm{free}}}1 + \sum_{d|e, d>1} \frac{\mu(d)}{\phi(d)} \sum_{x\in\Gamma_{d}} \sum_{\substack{n\leq x \\ n= \square-\textrm{free}}} \chi(n)\right\}.
\end{split}
\end{equation}

Next, let $k$ be a divisor of  $\mathrm{Rad}(p-1), $ the \emph{radical} of $p-1$ (i.e., the product of the distinct primes dividing $p-1$). Write
\begin{equation}\label{jeremy}
\mathrm{Rad}(p-1)=k p_1\cdots p_s,
\end{equation}
 where $1 \leq s \leq \omega(p-1)$ and $p_1, \ldots, p_s$ are distinct primes and the \emph{core} $k$ is the product of the $\omega(p-1)-s$
 smallest (distinct)  primes dividing $p-1$. We describe this situation as \emph{sieving with core $k$ and $s$ sieving primes}.
  A useful inequality for $N^{\square}(p,x)$ is now proved as in Lemma 3.1 of \cite{COT}.

 \begin{Lem}\label{helpusall}Given the prime $p$, assume that $(\ref{jeremy})$ holds. Then
 \begin{eqnarray*}N^{\square}(p,x) &  \geq &\sum_{i=1}^s N_{kp_i}^{\square}(p,x) - (s-1)N_{k}^{\square}(p,x)\\
 & =  &     \sum_{i=1}^s \left\{N_{kp_i}^{\square}(p,x)-\left(1- \frac{1}{p_i}\right) N_{k}^{\square}p,(x) \right\} + \delta N_{k}^{\square}(p,x),
 \end{eqnarray*}
 where
 \begin{equation}\label{corbyn}
 \delta= 1-\sum_{i=1}^s \frac{1}{p_i}.
 \end{equation}
\end{Lem}

Now, by  (\ref{ramsay}),  (\ref{attlee})  and Lemma \ref{james}, we have
\begin{equation}\label{tony}
N_{k}^{\square}(p,x)\geq \frac{\phi(k)}{k}( x^{1/2}p^{1/4}) \left\{\frac{6}{\pi^{2}}(  x^{1/2}p^{-1/4})  - \frac{0.104}{p^{1/4}}-2^{\omega(k)+1}(\log p)^{1/2}\right\}.
\end{equation}
Similarly, for each $i=1, \ldots, s$, since $\frac{\phi(kp_i)}{kp_i}= \left(1- \frac{1}{p_i}\right)\frac{\phi(k)}{k}$, we have
\begin{equation}\label{blair}
\left|N_{kp_i}^{\square}(p,x)-\left(1- \frac{1}{p_i}\right)N_{k}^{\square}(p,x)\right|  \leq \frac{\phi(k)}{k}( x^{1/2}p^{1/4})\left(1- \frac{1}{p_i}\right)2^{\omega(k)+1}(\log p)^{1/2}
\end{equation}
(using also the fact that $2^{\omega(kp_i)} - 2^{\omega(k)}= 2^{\omega(k)}$).

\begin{thm}\label{Harold} Given the prime $p$, assume that $(\ref{jeremy})$ holds.  Suppose that $\delta$ (defined by $(\ref{corbyn})$) is positive and set
\begin{equation*} 
\Delta = \frac{s-1}{\delta}+2.
\end{equation*}
Suppose also that
\begin{equation} \label{brown}
G_s(x):= x^{1/2} p^{-1/4}-\frac{\pi^{2}}{6} \left(\frac{0.104}{p^{1/4}} + 2^{\omega(k) +1}\Delta(\log p)^{1/2}\right)>0 .
\end{equation}
Then $p$ possesses a square-free primitive root less than $x$.
\end{thm}
\begin{proof} Apply $(\ref{tony})$ and $(\ref{blair})$ to Lemma $\ref{helpusall}$ and use the fact that $\sum_{i=1}^s\left(1-\frac{1}{p_i}\right)=s-1+\Delta$.

\end{proof}

\subsection{Application of Theorem \ref{Harold}}
As noted already, we need only consider $8\leq \omega(p-1) \leq 29$. For example, when $\omega(p-1) = 29$ the choice $s=26$ shows that (\ref{brown}) is true for all $p> 9.2\cdot 10^{18}$. Since $p-1$ has 29 distinct prime factors we have $p-1 \geq 2 \cdot 3 \cdots p_{29} > 2.8\cdot 10^{44}$, which proves Theorem \ref{Main} for $\omega(p-1) = 29$. We continue in this way, keeping in mind (\ref{lloyd}). Choosing $s= \omega(p-1) -3$ for  $14\leq \omega(p-1) \leq 29$ and $s=\omega(p-1) -2$ for  $8 \leq \omega(p-1) \leq 12$ completes the proof of Theorem \ref{Main} for every case except $\omega(p-1) = 13$.

When $\omega(p-1) =13$ we choose $s=10$, whence we require $p> 3.34\cdot 10^{15}$ to satisfy (\ref{brown}). Therefore one needs to check $p\in[2.5\cdot 10^{15}, 3.34\cdot 10^{15}]$.   Since $p$ is odd we have $2|p-1$; also, we must have $3|p-1$ since otherwise $p-1 \geq 2 \cdot 5 \cdots p_{14} > 4.3\cdot 10^{15}$, which satisfies (\ref{brown}). 

One can now proceed as in \cite[\S 3.1]{MTT}. We begin by noting that when $s=10$ we have $\delta > 1 - 1/7 - 1/11- \cdots - 1/p_{13} > 0.416$. If $5\nmid p-1$ then we have $\delta > 1 - 1/11  -\cdots - 1/p_{14} > 0.535$. This larger value of $\delta$ shows that (\ref{brown}) is true for all $p> 1.2\cdot 10^{15}$. We continue, with the choice $s=10$ to show that $5,\ldots, 19$ all divide $p-1$. Finally, with $s=11$ we conclude that $23|p-1$. 

Since $p-1 = m \cdot 2 \cdots 23$, and, since $2.5\cdot 10^{15} \leq p \leq 3.34\cdot 10^{15}$ we conclude that there are less than $4\cdot 10^{6}$ possible numbers which may not satisfy (\ref{brown}). Only 518 of these give rise to primes $p$ with $\omega(p-1) = 13$. For each of these possible exceptions we compute the \textit{exact} value of $\delta$, rather than merely a lower bound. We  then eliminate many of these cases by feeding these $\delta$'s into Theorem \ref{Harold}: this leaves a list of 25 possible exceptions the smallest of which is 
$$2,513,954,577,154,020.$$
Owing to the abundance of square-free primitive roots, we merely verify that each of these 25 numbers contains a square-free primitive root less than 100.

We note that our method allows one to prove $g^{\square}(p) < p^{\alpha}$ for all $p\geq p_{0}(\alpha)$, where $p_{0}(\alpha)$ is given explicitly, and where $\alpha>\frac{1}{2}$. For example, we are able to show that $g^{\square}(p) < p^{3/4}$  for all $p> 1.2\cdot 10^{34}$. 

We also note that in (\ref{clement}) we have used the
 P\'{o}lya--Vinogradov inequality
\begin{equation}\label{point}
\bigg| \sum_{M<n\leq N} \chi(n) \bigg| \leq \sqrt{p}\log p,
\end{equation}
valid for all non-principal characters $\chi(n)$ to the modulus $p$. The use of sharper versions of (\ref{point}) would improve (\ref{attlee}) and lead to an improvement in Theorem \ref{Main}.
Finally since $2$ is the only primitive root modulo $3$ one could not aim to produce a version of Theorem \ref{Main}, holding for all primes, with an exponent less than $\log 2/\log 3 = 0.6309\ldots$. 

\section{Square-full primitive roots}\label{conch}
Less is known about $g^{\blacksquare}(p)$, the number of square-full primitive roots modulo $p$. An integer $n$ is square-full if for all primes $l|n$ we have $l^{2}|n$. Let $N^{\blacksquare}(p,x)$ denote the number of square-full primitive roots modulo $p$ not exceeding $x$.  Shapiro \cite[p.\ 307]{Shapiro} proved that
\begin{equation}\label{vote}
N^{\blacksquare}(p,x) = \frac{\phi(p-1)}{p-1} \left\{ C x^{1/2} + O(2^{\omega(p-1)}  p^{1/6} (\log p)^{1/3} x^{1/3})\right\},
\end{equation}
where $C$ is an explicit constant. Liu and Zhang \cite[Thm 1.2]{LiuZhang05}   showed that
\begin{equation}\label{of}
N^{\blacksquare}(p,x) = \frac{\phi(p-1)}{p-1} \left\{ C x^{1/2} + O(p^{9/44 + \epsilon} x^{1/4 + \epsilon})\right\},
\end{equation}
whence $g^{\blacksquare}(p) \ll p^{9/11 + \epsilon}$. 

It would be interesting to estimate the size of $p_{0}$ such that $g^{\blacksquare}(p)< p$ for all $p\geq p_{0}$.
Shapiro's result (\ref{vote}) is insufficient to show this; Liu and Zhang's result (\ref{of}) is based on arguments out of which it would be difficult to derive explicit constants. We are grateful to Adrian Dudek who computed that $1,052,041$ is the largest prime $p$ with $p\leq 3\cdot 10^{6}$ that does not have a square-full primitive root less than $p$.


\begin{thebibliography}{99}

\bibitem{Cipu}
M.~Cipu
\newblock Further remarks on Diophantine quintuples.\newblock {\em Acta. Arith.}
\newblock 168(3), 201--219, 2015.


\bibitem{COT}
S.~D. Cohen, T.~Oliveira~e Silva, and T.~S. Trudgian.
\newblock On Grosswald's conjecture on primitive roots.
\newblock {\em Acta. Arith.}
\newblock 172(3), 263--270, 2016



\bibitem{Ha}
Ha, J.
\newblock On the least prime primitive root.
\newblock {\em J. Number Theory}
133(11), 3645--3669, 2013.



\bibitem{LiuZhang05}
Liu, H. and Zhang, W.
\newblock On the squarefree and squarefull numbers
\newblock {\em J. Math. Kyoto Univ.}, 45(2):247--255, 2005.

\bibitem{MTT}
K.~McGown, E.~Trevi\~{n}o, and T.~S. Trudgian.
\newblock Resolving Grosswald's conjecture on GRH.
\newblock {\em Funct. Approx.}
\newblock To appear.



\bibitem{Shapiro}
Shapiro, H. N.
\newblock {\em Introduction to the Theory of Numbers.} Pure and Applied Mathematics.
\newblock New York, John Wiley and Sons, Inc., 1983.


\bibitem{Shoup}
Shoup, V.
\newblock Searching for primitive roots in finite fields.
\newblock{\em  Math. Comp.} 58, 369--380, 1992.



\end{thebibliography}
\end{document}